\documentclass{amsart}
\usepackage{amssymb}
\usepackage{amsmath}
\usepackage{mathrsfs}
\usepackage{graphicx}
\usepackage{color}
\usepackage{url}
\usepackage{tikz}
\usetikzlibrary{cd}
\usepackage[ruled]{algorithm2e}
\usepackage{color}
\usepackage{todonotes}

\newcommand{\john}[1]{\todo[color=green!30]{#1 \\ \hfill --- J.}}

\theoremstyle{plain}
\newtheorem{theorem}{Theorem}[section]
\newtheorem{lemma}[theorem]{Lemma}
\newtheorem{proposition}[theorem]{Proposition}
\newtheorem{corollary}[theorem]{Corollary}

\theoremstyle{definition}

\newtheorem{question}[theorem]{Question}
\newtheorem{problem}[theorem]{Problem}

\newcommand{\A}{\mathcal{A}}
\newcommand{\U}{\mathcal{U}}
\DeclareMathOperator{\Spec}{Spec}
\DeclareMathOperator{\Trop}{Trop}
\DeclareMathOperator{\Mat}{Mat}
\newcommand{\x}{\mathbf{x}}
\newcommand{\y}{\mathbf{y}}

\newcommand{\bbZ}{\mathbb{Z}}
\newcommand{\bbP}{\mathbb{P}}
\newcommand{\bbA}{\mathbb{A}}
\renewcommand{\d}{\dagger}

\begin{document}

\title{Upper cluster algebras and choice of ground ring}
\author{Eric Bucher}
\address{Department of Mathematics, Michigan State University, East Lansing, MI 48824.}
\email{ebuche2@math.msu.edu}
\author{John Machacek}
\address{Department of Mathematics, Michigan State University, East Lansing, MI 48824.}
\email{machace5@math.msu.edu}
\author{Michael Shapiro}
\address{Department of Mathematics, Michigan State University, East Lansing, MI 48824.}
\email{}

\begin{abstract}
We initiate a study of the dependence of the choice of ground ring on the question on whether a cluster algebra is equal to its upper cluster algebra.
A condition for when there is equality of the cluster algebra and upper cluster algebra is given by using a variation of Muller's theory of cluster localization.
An explicit example exhibiting dependence on the ground ring is provided.
We also present a maximal green sequence for this example. 
\end{abstract}
\maketitle

\section{Introduction}
Cluster algebras were invented by Fomin and Zelevinsky to study the total positivity of Lusztig's canonical basis.
At its purest form what a cluster algebra, $\mathcal{A}$, does is attempt to study an infinite dimensional algebra by creating a combinatorics that separates the generators of $\mathcal{A}$ into \emph{seeds}.
These generators are referred to as \emph{cluster variables}.
The combinatorial data is a recipe for how to travel between the seeds.
This process from turning one seed into another is what is known as \emph{mutation}.
One of the most fascinating and powerful results in the study of such algebras is what is known as the Laurent phenomenon.
To paraphrase, this states that any cluster variable can be expressed as a Laurent polynomial in the variables from any seed in the cluster algebra.
The Laurent phenomenon leads to a very natural object called the upper cluster algebra, denoted $\mathcal{U}$.
Cluster algebras and upper cluster algebras will be precisely defined in Section~\ref{sec:prelim}.
This algebra, $\mathcal{U}$, is the intersection of the ring of Laurent polynomials for each seed. Due to the Laurent phenomenon we know that $\mathcal{A} \subseteq \mathcal{U}$.
At times we have more, and $\mathcal{A}=\mathcal{U}$.
This can be a powerful result.
Cluster structures can be observed in coordinate rings of many natural algebraic varieties.
Some examples include Grassmannians~\cite{Scott} and double Bruhat cells of complex simple Lie groups~\cite{BFZ}.
The upper cluster algebra can, as it does for Double Bruhat cells~\cite[Theorem 2.10]{BFZ}, coincide with the (homogeneous) coordinate ring of the variety.
When we have equality $\A = \U$, such coordinate rings can be approached by utilizing the explicit combinatorial description of the generators of $\A$.

Muller proves that $\mathcal{A}=\mathcal{U}$ for locally acyclic cluster algebras~\cite{MullerSIGMA}.
For a certain Cremmer-Gervais cluster algebra constructed by Gekhtman, Shapiro, and Vainshtein it was shown that $\mathcal{U} \neq \mathcal{A}$~\cite{GSV-CGshort, GSV-CGlong}. 
Further exploring these two results, we realized that the Cremmer-Gervais cluster algebra considered by Gekhtman, Shapiro, and Vainshtein is locally acyclic.
We reconciled these competing statements, by realizing that the cluster algebras had the same initial seed but were being generated over different \emph{ground rings}.
In this paper we wish to explore what impact the ground ring has on the question of whether or not $\mathcal{U}=\mathcal{A}$. 

In Section~\ref{sec:prelim} of this paper, we will present the definition of cluster algebras over general ground rings as well as the definition of the upper cluster algebra.
We will generalize some aspects of the theory of locally acyclic cluster algebras to more general ground rings in Section~\ref{s:locallyacyclic}.
A condition for $\A = \U$ is given in Theorem~\ref{thm:A=U}, and a condition implying $\A = \U$ for certain acyclic cluster algebras of geometric type in deduced from this theorem in Corollary~\ref{cor:A=Ugeo}.
The previously mentioned Cremmer-Gervais cluster algebra is considered in Section~\ref{s:Poisson}.
Proposition~\ref{prop:CG} uses this Cremmer-Gervais cluster algebra to given an explicit example illustrating the sensitivity of the $\A = \U$ question on the choice of ground ring.
In Section~\ref{s:reddening} we present a maximal green sequence for the quiver defining the Cremmer-Gervais cluster algebra from Section~\ref{s:Poisson}.
It is thought there may be a relationship between the existence of a maximal green sequence, or more generally a reddening sequence, and the equality $\A = \U$.
Our example shows any such relationship will be dependent on the chosen ground ring.
We conclude with Section~\ref{sec:conclusion} offering some potential directions for future work.

We think the results presented here are only a first step toward studying upper cluster algebras and the dependence on ground rings.
Also, we believe this is a natural topic to consider.
The reason for the choice of ground ring of the Cremmer-Gervais cluster algebra constructed by Gekhtman, Shapiro, and Vainshtein is to ensure that the cluster algebra constructed consists only of regular functions on $SL_3$.
The ground ring must be enlarged to apply Muller's result to conclude $\A = \U$.
Enlarging the ground ring forces one to consider regular functions not on $SL_3$, but rather on some open subset.
So, considering different ground rings is geometrically meaningful.
Furthermore, if $\A = \U$ this gives algebraic and geometric information by knowing that $\A$ is an intersection of Laurent polynomial rings.
For example, if the chosen ground ring is a normal domain, then $\U$ will be a normal domain.
This will be shown in Proposition~\ref{prop:normal}.
In this case $\Spec(\U)$ is a normal scheme; so, $\A$ and $\Spec(\A)$ will have these properties as well when $\A = \U$.

\section{Notation and preliminaries}
\label{sec:prelim}

Let $\mathbb{P}$ be a semifield. This means that $\mathbb{P}$ is a torsion free abelian group whose operation is written multiplicatively. Additionally, $\mathbb{P}$ is equipped with an auxiliary addition $\oplus$, which is commutative associative and distributive over the multiplication of $\mathbb{P}$. We now want to consider a \emph{ground ring}, $\bbZ \subseteq \mathbb{A} \subseteq\mathbb{Z}\mathbb{P}$. This is sometimes called the coefficient ring of the cluster algebra. 

Let $\mathcal{F}$ be a field which contains $\mathbb{Z}\mathbb{P}$. A \emph{seed} of rank $n$ in $\mathcal{F}$ is a triple $(\mathbf{x},\mathbf{y},B)$ consisting of three parts:

\begin{itemize}
\item the \emph{cluster} $\mathbf{x} = \{x_1,x_2, \dots , x_n\}$ is an $n$-tuple in $\mathcal{F}$ which freely generates $\mathcal{F}$ as a field over the fraction field of $\mathbb{Z}\mathbb{P}$,

\item the \emph{coefficients} $\mathbf{y} = \{y_1,y_2, \dots , y_n\}$ is an $n$-tuple in $\mathbb{P}$, and

\item the \emph{exchange matrix} $B$ is an integral, skew-symmetrizable $n \times n$ matrix. 

\end{itemize} 

A seed $(\mathbf{x},\mathbf{y}, B)$ may be mutated at an index $1 \leq k \leq n$, to produce a new seed $(\mu_k(\mathbf{x}),\mu_k(\mathbf{y}),\mu_k(B))$. We say that the seed has been mutated in the direction $k$. Meaning that there are $n$ different mutations that can be applied to a given seed. The mutation is given by the following rules:

\begin{itemize}

\item $\mu_k(\mathbf{x}) := \{x_1,x_2,\dots,x_{k-1},x_k'x_{k+1},\dots, x_n\}$, where
\[x_k':=\frac{y_k\prod x_j^{[B_{jk}]_+} + \prod x_j^{[-B_{jk}]_+}}{(y_k \oplus1)x_k}\]

\item $\mu_k(\mathbf{y}) : = \{y_1',y_2', \dots y_n'\},$ where
\[ y'_i:= 
\begin{cases} 
y_i^{-1} & \text{ if } i=k \\
y_i y_k^{[B_{ki}]_+}(y_k \oplus 1)^{-B_{ki}} & \text{ if } i \neq k\\
\end{cases}\]

\item $\mu_k(B)$ is defined by 

\[\mu_k(B)_{ij} = 
\begin{cases}
-B_{ij} & \text{ if } i=k \text{ or } j=k, \\
B_{ij} + \frac{1}{2}( | B_{ik}|B_{kj}+B_{ik}|B_{kj}|) & \text{ otherwise }\\
\end{cases}
\]

\end{itemize}

Notice that $\mu_k^2$ is the identity. We say that two seeds are \emph{mutation equivalent} if one can be obtained by a sequence of mutations, up to permuting the indices of the seed. 

Now that we have established the fundamental definitions of seeds and mutation we can define a \emph{cluster algebra}. Given a seed $(\mathbf{x},\mathbf{y},B)$ we will call the union of all the seeds which are mutation equivalent to $(\mathbf{x},\mathbf{y},B)$ a set of \emph{cluster variables} in $\mathcal{F}$. The \emph{cluster algebra}, $\mathcal{A}_{\bbA}(\mathbf{x},\mathbf{y},B)$, is the unital $\bbA$-subalgebra of $\mathcal{F}$ generated by the cluster variables. Notice that since we are allowed to freely mutate when generating the cluster variables, that two mutation equivalent seeds will generate the same cluster algebra. For this reason sometimes it is common to leave the seed out of the notation and simply refer to the algebra as $\mathcal{A}_{\bbA}$ or just $\A$ when the choice of ground ring is clear. 

\subsection{The Laurent phenomenon}
The Laurent phenomenon~\cite[Theorem 3.1]{FZI} is a very important property of cluster algebras. It states that if our ground ring for $\mathcal{A}$ is $\mathbb{Z}\mathbb{P}$, then we have that $\mathcal{A}$ is a subalgebra of $\mathcal{A}[x_1^{- 1},x_2^{- 1},\dots , x_n^{- 1}]=\mathbb{ZP}[x_1^{\pm 1},x_2^{\pm 1},\dots , x_n^{\pm 1}]$.

In the initial work of Fomin and Zelevinsky they give a more generalized Laurent phenomenon, and outline conditions that allow for the Laurent phenomenon to hold even if we work over a potentially smaller ground ring $\mathbb{A} \subsetneq \mathbb{ZP}$~\cite[Theorem 3.2]{FZI}. Throughout this paper, when we discuss the cluster algebra over a ground ring $\mathbb{A}$, we will assume that the Laurent phenomenon holds over $\bbA$ and that $\bbA$ contains all coefficients appearing in exchange relations. In this case we have:
\begin{equation}
\mathcal{A} \hookrightarrow \mathcal{A}[x_1^{- 1},x_2^{- 1},\dots , x_n^{- 1}]=\mathbb{A}[x_1^{\pm 1},x_2^{\pm 1},\dots , x_n^{\pm 1}]
\label{conditionLP}
\end{equation}
and
\begin{equation}
\frac{y_i}{1 \oplus y_i}, \frac{1}{1 \oplus y_i} \in \bbA
\label{conditionCoeff}
\end{equation}
for any seed $(B,\x,\y)$ of the cluster algebra.

\subsection{The upper cluster algebra}
Now we will consider an algebra closely related to the $\mathcal{A}$. This is the \emph{upper cluster} algebra denoted by $\mathcal{U}$ or by $\U_{\bbA}((B,\x,\y))$. 

We define this as follows:

\[ \mathcal{U} := \bigcap_{\mathbf{x}\in \mathcal{A}} \mathbb{A}[x_1^{\pm 1},x_2^{\pm 1},\dots , x_n^{\pm 1}]. \]

Since we have chosen a ground ring, $\mathbb{A}$, where the criteria for the Laurent phenomenon are met, we see that injections from each seed mean that $\mathcal{A} \subseteq \mathcal{U}$. There are many occurrences where $\mathcal{A}=\mathcal{U}$ which gives an alternative way of understanding the cluster algebra $\A$.
The choice of ground ring plays a substantial role in deciding whether or not we are in this nice situation. In the rest of this paper we will explore exactly how the choice of ground ring impacts whether $\mathcal{A}=\mathcal{U}$.

We now offer a proposition on normality generalizing~\cite[Proposition 2.1]{MullerAdv} to our situation.
Recall, an integral domain is called a \emph{normal domain} if it is integrally closed in its field of fractions.
A semifield $\bbP$ is always torsion-free, and thus any $\bbZ \subseteq \bbA \subseteq \bbZ\bbP$ is an integral domain.

\begin{proposition}
If $\bbA$ is a normal domain, then $\U_{\bbA}$ is a normal domain.
\label{prop:normal}
\end{proposition}
\begin{proof}
Assume $\bbA$ is a normal domain.
From the definition it follows that the intersection of normal domains is again a normal domain.
So, it suffices to show that $\bbA[x_1^{\pm 1}, x_2^{\pm 1}, \dots, x_n^{\pm 1}]$ is a normal domain for any seed $(B, \x, \y)$.
This is true since any polynomial ring over a normal domain is a normal domain~\cite[Lemma 10.36.8]{Stacks} and any localization of a normal domain is a normal domain~\cite[Lemma 10.36.5]{Stacks}.
\end{proof}

\subsection{Geometric type}

The \emph{tropical semifield} $\Trop(z_1, z_2, \dots, z_m)$ is the free abelian group generated by $z_1, z_2, \dots, z_m$ with auxiliary addition given by
\[\prod_{i=1}^m z_i^{a_i} \oplus \prod_{i=1}^m z_i^{b_i} = \prod_{i=1}^m z_i^{\min(a_i,b_i)}.\]
A cluster algebra is said to be of \emph{geometric type} if it is defined over a tropical semifield.
In this case $\bbZ\bbP = \bbZ[z_1^{\pm 1}, z_2^{\pm 1}, \dots, z_m^{\pm 1}]$ is the Laurent polynomial ring in $z_1, z_2, \dots, z_m$.
We are particularly interested in the polynomial ground ring $\bbZ[z_1, z_2, \dots, z_m]$ which we will denote by $\bbZ\bbP_+$.
For cluster algebras of geometric type there is a sharpening of the Laurent phenomenon giving $\A \subseteq \bbZ\bbP_+[x_1^{\pm 1}, x_2^{\pm 1}, \dots, x_n^{\pm 1}]$.
This sharpening of the Laurent phenomenon can be found in~\cite[Theorem 3.3.6]{FWZ}.

A \emph{quiver} is a directed graph without loops or directed $2$-cycles, but parallel arrows are allowed.
When $(B, \x, \y)$ is a seed for a cluster algebra of geometric type and the matrix $B$ is skew-symmetric, we will often consider a quiver which is equivalent data to the seed $(B, \x, \y)$.
The $n \times n$ skew-symmetric matrix $B$ is considered as a signed adjacency matrix of a quiver.
That is, the quiver $Q$ has $B_{ji}$ arrows $i \to j$ where negative arrows correspond to reversing the direction.
The generators $z_1, z_2, \dots, z_m$ of $\bbP$ correspond to additional vertices of $Q$ called \emph{frozen vertices}.
Non-frozen vertices are known as \emph{mutable vertices}.
Mutable vertices of the quiver will be depicted as circles while frozen vertices are represented by squares.
There are no arrows between frozen vertices.
Arrows between mutable and frozen vertices are obtained by considering $\y$.
If $y_i = z_1^{a_i}z_2^{a_2} \cdots z_m^{a_m}$, then the mutable vertex $i$ has $a_j$ arrows $i \to z_j$.
As an example if $\bbP = \Trop(z_1, z_2)$ while our seed has $\y = (z_1z_2^{-1}, z_2^{-1})$ and
\[B=\begin{bmatrix}
0 & -2\\
2 & 0
\end{bmatrix} \]
the corresponding quiver $Q$ would be the quiver pictured in Figure~\ref{fig:example}.

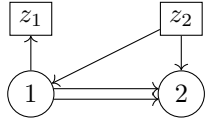
\begin{figure}
\begin{tikzpicture}
\node at (0,0) [circle, draw] (1) {$1$};
\node at (2,0) [circle, draw] (2) {$2$};
\node at (0,1) [rectangle, draw] (z1) {$z_1$};
\node at (2,1) [rectangle, draw] (z2) {$z_2$};
\draw[->] (0.3,0.07) to (1.7,0.07);
\draw[->] (0.3,-0.07) to (1.7,-0.07);
\draw[->] (1) to (z1);
\draw[->] (z2) to (1);
\draw[->] (z2) to (2);
\end{tikzpicture}
\caption{A quiver.}
\label{fig:example}
\end{figure}

\section{Locally isolated and Locally acyclic cluster algebras}
\label{s:locallyacyclic}
We now review Muller's notion of locally acyclic cluster algebras~\cite{MullerAdv}.
This section closely follows~\cite{MullerSIGMA} while describing some changes needed to adapt the theory of locally acyclic cluster algebras for ground rings other than $\bbZ\bbP$.
Let $(B,\x,\y)$ be a seed of rank $n$ and $\bbA$ be a ground ring satisfying (\ref{conditionLP}) and (\ref{conditionCoeff}).
Denote the corresponding cluster algebra by $\A = \A_{\bbA}(B,\x,\y)$ and upper cluster algebra by $\U = \U_{\bbA}(B,\x,\y)$.

The \emph{freezing} of $\A$ at $x_n \in \x$ is the cluster algebra $\A^{\dagger} = \A_{\bbA^{\d}}(B^{\d}, \x^{\d}, \y^{\d})$ defined as follows
\begin{itemize}
\item The new semifield is $\bbP^{\d} = \bbP \times \bbZ$  with $x_n$ as the generator of the free abelian group $\bbZ$. The auxiliary addition is extended as
\[(p_1 x_n^a) \oplus (p_2 x_n^b) = (p_1 \oplus p_2)x_n^{\min(a,b)}.\]
\item The new ground ring $\bbA^{\d} \subseteq \bbZ\bbP^{\d}$ is $\bbA^{\d} = \bbA[x_n^{\pm 1}]$.
\item The new ambient field is $\mathcal{F}^{\d} = \mathbb{Q}(\bbP^{\d}, x_1, x_2, \dots, x_{n-1})$ and the new cluster is $\x^{\d} = (x_1, x_2, \dots, x_{n-1})$.
\item The new coeffients are $\y^{\d} = (y_1^{\d}, y_2^{\d}, \dots, y_{n-1}^{\d})$ where $y_i^{\d} = y_i x_n^{B_{ni}}$.
\item The new exchange matrix $B^{\d}$ is obtained from $B$ by deleting the $n$th row and column.
\end{itemize}
In this setting the upper cluster algebra corresponding to $\A^{\d}$ will be denoted $\U^{\d}$.
Notice in a cluster algebra arising from a quiver that a freezing exactly corresponds to replacing the mutable vertex labeled by $x_n$ with a frozen vertex.
By considering a permutation of indices we can freeze at any $x_i \in \x$.
We can freeze at some subset of cluster variables by iteratively  freezing at individual cluster variables.
This process is independent of the order of freezing.
A freezing $\A^{\d}$ of $\A$ at $\{x_{i_1}, x_{i_2}, \dots, x_{i_m}\} \in \x$ is called a \emph{cluster localization} if $\A^{\d} = \A[(x_{i_1}x_{i_2}\cdots x_{i_m})^{-1}]$.
A \emph{cover} of $\A$ is a collection $\{A_i\}_{i \in I}$ of cluster localizations such that for any prime ideal $P \subseteq \A$ there exists $i \in I$ where $\A_iP \subsetneq \A_i$.
Lemma~\ref{lem:freezelocal} below is not exactly~\cite[Lemma 1]{MullerSIGMA}, but follows immediately and will be all that is used for our purposes.

\begin{lemma}[{\cite[Lemma 1]{MullerSIGMA}}]
If $\A^{\d}$ is a freezing of a cluster algebra $\A$ at cluster variables $\{x_{i_1}, x_{i_2}, \dots, x_{i_m}\}$ and $\A^{\d} = \U^{\d}$, then $\A^{\d} = \A[(x_{i_1}, x_{i_2}, \dots, x_{i_m})^{-1}]$ is a cluster localization.
\label{lem:freezelocal}
\end{lemma}

The seed $(B, \x, \y)$ is called \emph{isolated} if $B$ is the zero matrix.
A cluster algebra defined by an isolated seed is also referred to as \emph{isolated}.
In terms of quivers, isolated means that there are no arrows between mutable vertices.
The seed $(B, \x, \y)$ is said to be \emph{acyclic} if there are not $i_1, i_2, \dots, i_{\ell} \in \{1,2,\dots, n\}$ with $B_{i_{j+1} i_j} > 0$ for $1 \leq j < \ell$ and $i_1 = i_{\ell}$.
A cluster algebra defined by an acyclic seed is called an acyclic cluster algebra.
A \emph{locally isolated}, respectively \emph{locally acyclic}, cluster algebra is a cluster algebra for which there exists a cover by isolated, respectively acyclic, cluster algebras.

In~\cite{MullerSIGMA} the only ground ring considered is $\bbZ\bbP$; however, proofs of the following results go through without change for more general ground rings.

\begin{lemma}[{\cite[Lemma 2]{MullerSIGMA}}]
Let $\{\A_i\}_{i\in I}$ be a cover of $\A$. If $\A_i=\U_i$ for all $i \in I$, then $\A = \U$.
\label{lem:Ucover}
\end{lemma}

\begin{proposition}[{\cite[Proposition 3]{MullerSIGMA}}]
Let $\A$ be an isolated cluster algebra. Then $\A = \U$.
\label{prop:isolated}
\end{proposition}

This immediately gives the following result.

\begin{theorem}
If $\A$ is a locally isolated cluster algebra, then $\A = \U$.
\label{thm:locallyisolated}
\end{theorem}

The definition of a locally isolated cluster algebra and Theorem~\ref{thm:locallyisolated} are not explicitly stated in~\cite{MullerSIGMA}.
This is because over the ground ring $\bbZ\bbP$ being locally isolated is equivalent to being locally acyclic as every acyclic cluster algebra over $\bbZ\bbP$ admits a cover by isolated cluster algebras~\cite[Proposition 4]{MullerSIGMA}.
The equivalence is not true over other ground rings.

In Section~\ref{s:Poisson} we give an example of a cluster algebra of geometric type which is locally acyclic over $\bbZ\bbP$, but for which the cluster algebra and upper cluster algebra do not coincide over a different natural choice of ground ring.
We show this example cluster algebra is locally acyclic by using Muller's Banff algorithm~\cite[Theorem 5.5]{MullerAdv}.
A pair of vertices $(i_1, i_2)$ in a quiver $Q$ is called a \emph{covering pair} if $(i_1, i_2)$ is an arrow of $Q$ that is not contained in any bi-infinite path of mutable vertices.
The notion of covering pair is needed to run the Banff algorithm.
The (reduced) Banff algorithm can be found as Algorithm~\ref{alg:Banff} in this article.
The reduced version of the algorithm deletes vertices rather then freezes them.
This makes for a simpler check that a cluster algebra is locally acyclic, but has the down side that it does not compute the actual cover.

\begin{algorithm}
\SetKwInOut{Input}{Input}
\SetKwInOut{Output}{Output}
\Input{A (mutable part of a) quiver $Q_0$}
\Output{A finite set $\mathbf{A}$ of acyclic quivers; or else the algorithm has \textbf{fails}.}
$\mathbf{A} \leftarrow \emptyset$, $\mathbf{B} \leftarrow \{Q_0\}$\;
\While{$\mathbf{B} \neq \emptyset$}{
Remove a quiver $Q$ from $\mathbf{B}$\;
\uIf{$Q$ is mutation equivalent to an acyclic quiver}{
$\mathbf{A} \leftarrow \mathbf{A} \cup \{Q\}$\;}
\uElseIf{$Q$ is mutation equivalent to a quiver $Q'$ with a covering pair $(i_1, i_2)$}{
Let $Q_j$ be the quiver obtained from $Q'$ by freezing (deleting) $i_j$ for $j = 1,2$\;
$\mathbf{B} \leftarrow \mathbf{B} \cup \{Q_1, Q_2\}$\;}
\uElse{The algorithm \textbf{fails}.}
}
\caption{The (reduced) Banff algorithm to determine if a quiver is locally acyclic.}
\label{alg:Banff}
\end{algorithm}

In the remainder of this section we analyze the proof of~\cite[Proposition 4]{MullerSIGMA} and consider conditions which imply a cluster algebra is locally isolated.
This allows the application of Theorem~\ref{thm:locallyisolated} to conclude the cluster algebra is equal to its upper cluster algebra.
An index $i \in \{1,2,\dots, n\}$ is a \emph{source} in the seed $(B, \x, \y)$ if $B_{ki} \geq 0$ for all $1 \leq k \leq n$.
In this case mutation in the direction $i$ gives
\begin{equation} x_ix'_i = \frac{y_i}{y_i \oplus 1}\prod_{B_{ki} > 0} x_k^{B_{ki}} + \frac{1}{y_i \oplus 1}.
\label{eq:sourcemut}
\end{equation}
A key step in showing an acyclic cluster algebra over $\bbZ\bbP$ is covered by isolated cluster algebras is as follows\footnote{The argument that follows is the ``source version'' and a corresponding ``sink version'' holds with corresponding modification. The ``sink version'' is used in~\cite{MullerSIGMA}.}.
Let $i$ be a non-isolated source choose $j$ with $B_{ji} > 0$, then after multiplying by $y_i \oplus 1$ and rearranging~(\ref{eq:sourcemut}) we obtain
\begin{equation}
((y_i \oplus 1) x'_i)x_i - \left(y_i \prod_{\substack{B_{ki} > 0\\k \neq j}} x_k^{B_{ki}}\right)x_j^{B_{ji}} = 1.
\label{eq:sourcecover}
\end{equation}
It is implied by~(\ref{eq:sourcecover}) that $1 \in \A x_i + \A x_j$ and it follows that $\{\A [x_i^{-1}], \A [x_j^{-1}]\}$ cover $\A$ provided these are cluster localizations.
Here we see that we do not necessarily need to be working over $\bbZ\bbP$.
The essential fact is that $1 \in \A x_i + \A x_j$.
This leads to the following definition.
Call the seed $(B,\x,\y)$ a \emph{source freezing seed} with respect to a ground ring $\bbA$ if $y_i \oplus 1\in \bbA$ for all $1 \leq i \leq n$.

\begin{lemma}
Let $(B,\x,\y)$ be a source freezing seed with respect to $\bbA$.
If $i$ is a source and $B_{ji} > 0$, then $\{\A[x_i^{-1}], \A[x_j^{-1}]\}$ cover $\A$ provided they are cluster localizations.
\label{lem:cover}
\end{lemma}
\begin{proof}
In the case $(B,\x,\y)$ a \emph{source freezing seed} with respect to $\bbA$, $i$ is a source, and $B_{ji} > 0$ we have that $1 \in \A x_i + \A x_j$ by (\ref{eq:sourcecover}).
Given any prime $\A$-ideal $P$ we must have either $x_i \not\in P$ or $x_j \not\in P$ since $P \neq \A$ is a proper ideal.
If $x_i \not\in P$, then $\A[x_i^{-1}]P \subsetneq \A[x_i^{-1}]$.
If $x_j \not\in P$, then $\A[x_j^{-1}]P \subsetneq \A[x_j^{-1}]$.
Thus if $\A[x_i^{-1}]$ and $\A[x_j^{-1}]$ are cluster localizations they form a cover.
\end{proof}

\begin{lemma}
Let $(B,\x,\y)$ be a source freezing seed, and let $(B^{\d}, \x^{\d}, \y^{\d})$ the freezing at $x_i$ for any $1 \leq i \leq n$.
The seed $(B^{\d}, \x^{\d}, \y^{\d})$ is a source freezing seed with respect to $\A^{\d} = \A[x_i^{\pm 1}]$.
\label{lem:seed}
\end{lemma}
\begin{proof}
Since $(B,\x,\y)$ is a source freezing seed, $1 \oplus y_j \in \bbA$ for each $1 \leq j \leq n$.
We must show that  $1 \oplus y^{\d}_j \in \bbA^{\d}$ for each $1 \leq j \leq n$ where $y^{\d}_j = y_j x_i^{B_{ij}}$.
We compute
\begin{align*}
1 \oplus y^{\d}_j &= 1 \oplus y_j x_i^{B_{ij}}\\
&= (1 \oplus y_j)x_i^{\min(0,B_{ij})} \in \bbA^{\d} = \bbA[x_i^{\pm 1}],
\end{align*}
and the lemma is proven.
\end{proof}

\begin{theorem}
If $\A = \A_{\bbA}(B,\x,\y)$ where $(B,\x,\y)$ is an acyclic source freezing seed with respect to $\bbA$, then $\A = \U$.
\label{thm:A=U}
\end{theorem}
\begin{proof}
Assume $(B, \x, \y)$ is an acyclic source freezing on rank $n$.
We will induct on the rank.
If $n \leq 1$, we are done by Proposition~\ref{prop:isolated} since $B$ must be isolated.
Now we may assume the seed in non-isolated, otherwise we are done by Proposition~\ref{prop:isolated}.
Since the seed in non-isolated and acyclic there must be a non-isolated source.
We choose a non-isolated source $i$ and then pick $j$ with $B_{ji} > 0$.
Freezing at $x_i$, the seed $(B^{\d}, \x^{\d}, \y^{\d})$ is an acyclic source freezing seed with respect to $\bbA^{\d}$ by Lemma~\ref{lem:seed}, and hence $\A^{\d} = \U^{\d}$ by induction since $\A^{\d}$ is of rank $n-1$.
Similarly freezing at $x_j$, the seed  $(B^{\d\d}, \x^{\d\d}, \y^{\d\d})$ is an acyclic source freezing seed with respect to $\bbA^{\d\d}$ by Lemma~\ref{lem:seed} and $\A^{\d\d} = \U^{\d\d}$ also.
Now $\A^{\d}$ and $\A^{\d\d}$ both must be cluster localizations by Lemma~\ref{lem:freezelocal}.
Thus Lemma~\ref{lem:cover} says that $\{\A^{\d}, \A^{\d\d}\}$ is a cover of $\A$.
We conclude $\A = \U$ using Lemma~\ref{lem:Ucover}.
\end{proof}

Theorem~\ref{thm:A=U} can be used to produce example of geometric type with $\A = \U$ over the polynomial ground ring $\bbZ\bbP_+$.
Call a quiver a \emph{source freezing quiver} if all arrows involving frozen vertices are directed from a mutable vertex to a frozen vertex.
An acyclic source freezing quiver is pictured in Figure~\ref{fig:asfquiver}.
Combining Theorem~\ref{thm:A=U} with the sharpening of the Laurent phenomenon~\cite[Theorem 3.3.6]{FWZ} for cluster algebras of geometric type we obtain the following corollary since any source freezing quiver defines a source freezing seed with respect to $\bbZ\bbP_+$.

\begin{corollary}
If $Q$ is an acyclic source freezing quiver defining a cluster algebra $\A$, then $\A = \U$ over the ground ring $\bbZ\bbP_+$.
\label{cor:A=Ugeo}
\end{corollary}

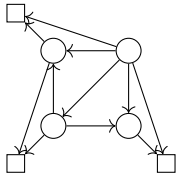
\begin{figure}
\begin{tikzpicture}
\node at (1,-1) [rectangle,draw] (2') {};
\node at (-1,-1) [rectangle,draw] (3') {};
\node at (-1,1) [rectangle,draw] (4') {};

\node at (0.5,0.5) [circle,draw] (1) {};
\node at (0.5,-0.5) [circle,draw] (2) {};
\node at (-0.5,-0.5) [circle,draw] (3) {};
\node at (-0.5,0.5) [circle,draw] (4) {};

\draw[->] (1) to (2');
\draw[->] (2) to (2');
\draw[->] (3) to (3');
\draw[->] (4) to (3');
\draw[->] (4) to (4');
\draw[->] (1) to (4');

\draw[->] (1) to (2);
\draw[->] (3) to (2);
\draw[->] (1) to (4);
\draw[->] (3) to (4);
\draw[->] (1) to (3);
\end{tikzpicture}

\caption{An acyclic source freezing quiver.}
\label{fig:asfquiver}
\end{figure}

\section{A Poisson-Lie example}
\label{s:Poisson}

Gekhtman, Shapiro, and Vainshtein have developed a program for studying cluster algebras through Poisson structures compatible with cluster mutation~\cite{GSV-book}.
In this program a cluster algebra $\A$ produces a \emph{log-canonical} coordinate system on the so called \emph{cluster manifold} which is a certain nonsingular part of $\Spec(\A)$.
Log-canonical coordinates are particularly nice coordinates which can be thought of as an algebraic analog of \emph{Darboux coordinates}~\cite{log-can}.

A cluster structure in the field of rational functions of an algebraic variety is called \emph{regular} if cluster variables are regular functions.
A main question is to construct explicitly a compatible regular cluster structure corresponding to a given variety equipped with an algebraic Poisson structure.
For a simple complex Lie group, Belavin and Drinfled have given a classification of Poisson-Lie structures coming from classical $R$-matrices~\cite{BD}.
The main conjecture of Gekhtman, Shapiro, and Vainshtein on cluster algebras and Poisson geometry states that for complex simple Lie groups the classification of regular cluster structures parallels the Belavin-Drinfled classification~\cite[Conjecture 3.2]{GSV-BD}. 

The cluster structure on double Bruhat cells from~\cite{BFZ} is known to correspond to the \emph{standard} Poisson-Lie structure~\cite[Chapter 4.3]{GSV-book}.
Other Poisson-Lie compatible cluster structures are called \emph{exotic}.
In certain cases Eisner has verified the conjecture and produced exotic cluster structures~\cite{E1,E2,E3}.
We will be particularly interested in an exotic cluster structure constructed by Gekhtman, Shapiro, and Vainshtein known as the \emph{Cremmer-Gervais} cluster structure on $SL_n$~\cite{GSV-CGshort, GSV-CGlong}.
This cluster structure extends to a cluster structure on the set of $n \times n$ matrices denoted $\Mat_n$.
We let $\A_{\bbA}(CG,n)$ and $\U_{\bbA}(CG, n)$ denote the cluster algebra and upper cluster algebra over the ground ring $\bbA$ coming from the Cremmer-Gervais cluster structure on $\Mat_n$.
The quiver $Q(CG,n)$ defines this cluster algebra.
Figure~\ref{fig:CG3} shows the quiver $Q(CG,3)$.
It is known that for any $n$ the upper cluster algebra $\U_{\bbZ\bbP_+}(CG,n)$ is naturally isomorphic to the ring of regular functions on $\Mat_n$~\cite[Theorem 3.1]{GSV-CGlong}.
The following result shows the sensitivity of the $\A = \U$ question on the ground ring

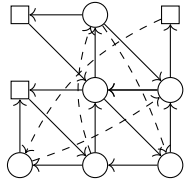
\begin{figure}
\begin{tikzpicture}
\node at (1,3) [rectangle, draw] (13) {};
\node at (2,3) [circle, draw] (23) {};
\node at (3,3) [rectangle, draw] (33) {};
\node at (1,2) [rectangle, draw] (12) {};
\node at (2,2) [circle, draw] (22) {};
\node at (3,2) [circle, draw] (32) {};
\node at (1,1) [circle, draw] (11) {};
\node at (2,1) [circle, draw] (21) {};
\node at (3,1) [circle, draw] (31) {};

\draw[->] (11) -- (12);
\draw[->] (12) -- (21);
\draw[->] (21) -- (11);
\draw[->] (21) -- (22);
\draw[->] (22) -- (12);
\draw[->] (32) -- (22);
\draw[->] (22) -- (31);
\draw[->] (31) -- (32);
\draw[->] (31) -- (21);
\draw[->] (22) -- (23);
\draw[->] (23) -- (32);
\draw[->] (32) -- (22);
\draw[->] (23) -- (13);
\draw[->] (13) -- (22);
\draw[->] (32) -- (33);

\draw[->, dashed] (11) [bend right=8]to (32);
\draw[->, dashed] (21) [bend left=20]to (23);
\draw[->, dashed] (23) [bend left=8]to (31);
\draw[->, dashed] (33) [bend right=20]to (11);

\end{tikzpicture}
\caption{The quiver $Q(CG,3)$.}
\label{fig:CG3}
\end{figure}

\begin{proposition}
We have equality $\A_{\bbZ\bbP}(CG,3) = \U_{\bbZ\bbP}(CG,3)$ over $\bbZ\bbP$ but strict containment $\A_{\bbZ\bbP_+}(CG,3) \subsetneq \U_{\bbZ\bbP_+}(CG,3)$ over $\bbZ\bbP_+$.
\label{prop:CG}
\end{proposition}

\begin{proof}
The equality $\A_{\bbZ\bbP}(CG,3) = \U_{\bbZ\bbP}(CG,3)$ follows from the fact that $\A_{\bbZ\bbP}(CG,3)$ is a locally acyclic cluster algebra over $\bbZ\bbP$.
This can be checked by applying the Banff algorithm.
A visual representation of the reduced Banff algorithm applied to $Q(CG,3)$ is shown in Figure~\ref{fig:Banff}.
In Figure~\ref{fig:Banff} mutation equivalence is denoted by $\Leftrightarrow$ and covering pairs used are displayed in thick red.
The strict containment $\A_{\bbZ\bbP}(CG,3) \subsetneq \U_{\bbZ\bbP}(CG,3)$ is~\cite[Theorem 4.1]{GSV-CGshort}.
\end{proof}

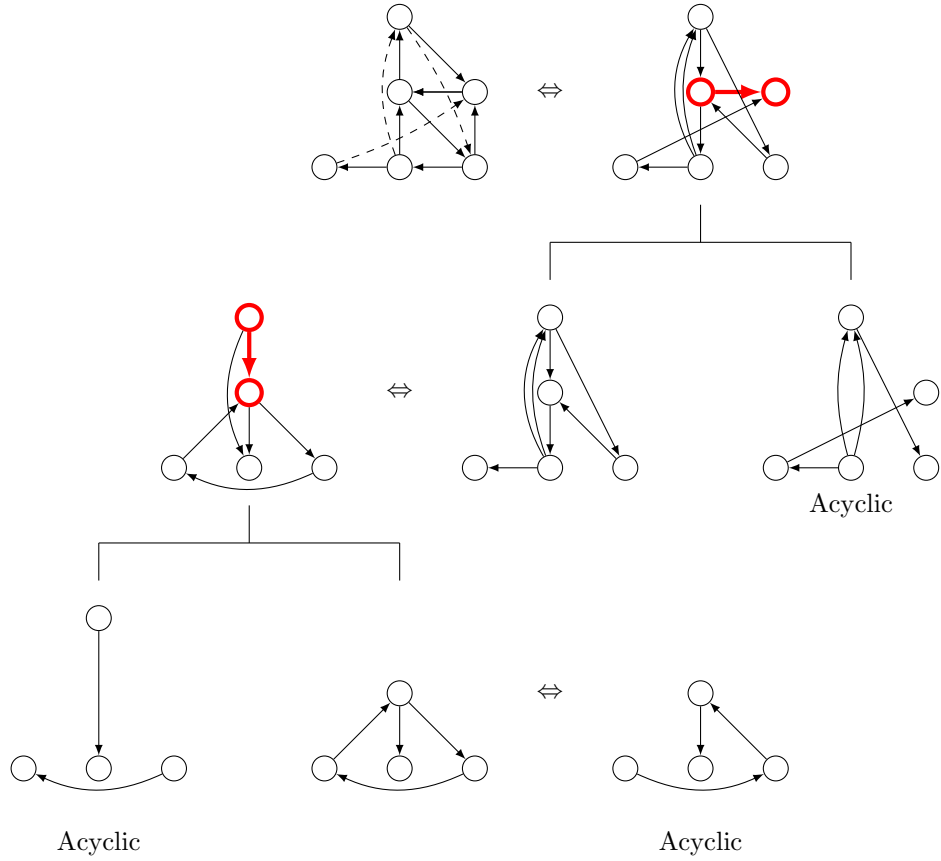
\begin{figure}
\centering
\begin{tikzpicture}
\node at (2,3) [circle, draw] (23) {};
\node at (2,2) [circle, draw] (22) {};
\node at (3,2) [circle, draw] (32) {};
\node at (1,1) [circle, draw] (11) {};
\node at (2,1) [circle, draw] (21) {};
\node at (3,1) [circle, draw] (31) {};

\draw[-{latex}] (21) -- (11);
\draw[-{latex}] (21) -- (22);
\draw[-{latex}] (32) -- (22);
\draw[-{latex}] (31) -- (32);
\draw[-{latex}] (31) -- (21);
\draw[-{latex}] (22) -- (23);
\draw[-{latex}] (23) -- (32);
\draw[-{latex}] (32) -- (22);
\draw[-{latex}] (22) -- (31);

\draw[-{latex}, dashed] (11) [bend right=8]to (32);
\draw[-{latex}, dashed] (21) [bend left=20]to (23);
\draw[-{latex}, dashed] (23) [bend left=8]to (31);

\node at (4,2) (equiv) {$\Leftrightarrow$};

\node at (6,3) [circle, draw] (63) {};
\node at (6,2) [circle, draw, ultra thick, red] (62) {};
\node at (7,2) [circle, draw, ultra thick, red] (72) {};
\node at (5,1) [circle, draw] (51) {};
\node at (6,1) [circle, draw] (61) {};
\node at (7,1) [circle, draw] (71) {};

\draw[-{latex}] (51) -- (72);
\draw[-{latex}] (61) -- (51); 
\draw[-{latex}] (63) -- (62);
\draw[-{latex}] (62) -- (61);
\draw[-{latex}, ultra thick, red] (62) -- (72);
\draw[-{latex}] (71) -- (62);
\draw[-{latex}] (63) -- (71);

\draw[-{latex}] (61) [bend left=20]to (63);
\draw[-{latex}] (61) [bend left=30]to (63);

\draw (6,0.5) -- (6,0);
\draw (4,0) -- (8,0);
\draw (4,0) -- (4,-0.5);
\draw (8,0) -- (8,-0.5);

\node at (8,-1) [circle, draw] (8-1) {};
\node at (9,-2) [circle, draw] (9-2) {};
\node at (7,-3) [circle, draw] (7-3) {};
\node at (8,-3) [circle, draw] (8-3) {};
\node at (9,-3) [circle, draw] (9-3) {};
\node at (8,-3.5) (acyclic) {Acyclic};

\draw[-{latex}] (8-3) -- (7-3);
\draw[-{latex}] (8-3) [bend left=15]to (8-1);
\draw[-{latex}] (8-3) [bend right=15]to (8-1);
\draw[-{latex}] (7-3) -- (9-2);
\draw[-{latex}] (8-1) -- (9-3);

\node at (4,-1) [circle, draw] (4-1) {};
\node at (4,-2) [circle, draw] (4-2) {};
\node at (3,-3) [circle, draw] (3-3) {};
\node at (4,-3) [circle, draw] (4-3) {};
\node at (5,-3) [circle, draw] (5-3) {};

\draw[-{latex}] (4-3) -- (3-3);
\draw[-{latex}] (4-3) [bend left=20]to (4-1);
\draw[-{latex}] (4-3) [bend left=30]to (4-1);
\draw[-{latex}] (4-1) -- (5-3);
\draw[-{latex}] (4-1) -- (4-2);
\draw[-{latex}] (4-2) -- (4-3);
\draw[-{latex}] (5-3) -- (4-2);

\node at (2,-2) (equiv) {$\Leftrightarrow$};

\node at (0,-1) [circle, draw, ultra thick, red] (0-1) {};
\node at (0,-2) [circle, draw, ultra thick, red] (0-2) {};
\node at (-1,-3) [circle, draw] (-1-3) {};
\node at (0,-3) [circle, draw] (0-3) {};
\node at (1,-3) [circle, draw] (1-3) {};

\draw[-{latex},ultra thick,red] (0-1) -- (0-2);
\draw[-{latex}] (0-2) -- (0-3);
\draw[-{latex}] (0-1) [bend right=25]to (0-3);
\draw[-{latex}] (1-3) [bend left=25]to (-1-3);
\draw[-{latex}] (-1-3) -- (0-2);
\draw[-{latex}] (0-2) -- (1-3);

\draw (0,-3.5) -- (0,-4);
\draw (-2,-4) -- (2,-4);
\draw (-2,-4) -- (-2,-4.5);
\draw (2,-4) -- (2,-4.5);

\node at (-2,-5) [circle, draw] (-2-5) {};
\node at (-3,-7) [circle, draw] (-3-7) {};
\node at (-2,-7) [circle, draw] (-2-7) {};
\node at (-1,-7) [circle, draw] (-1-7) {};
\node at (-2,-8) (acyclic) {Acyclic};

\draw[-{latex}] (-2-5) -- (-2-7);
\draw[-{latex}] (-1-7) [bend left=25]to (-3-7);

\node at (2,-6) [circle, draw] (2-6) {};
\node at (1,-7) [circle, draw] (1-7) {};
\node at (2,-7) [circle, draw] (2-7) {};
\node at (3,-7) [circle, draw] (3-7) {};

\draw[-{latex}] (1-7) -- (2-6);
\draw[-{latex}] (2-6) -- (3-7);
\draw[-{latex}] (3-7) [bend left=25]to (1-7);
\draw[-{latex}] (2-6) -- (2-7);

\node at (4,-6) (equiv) {$\Leftrightarrow$};

\node at (6,-6) [circle, draw] (6-6) {};
\node at (5,-7) [circle, draw] (5-7) {};
\node at (6,-7) [circle, draw] (6-7) {};
\node at (7,-7) [circle, draw] (7-7) {};
\node at (6,-8) (Acyclic) {Acyclic};

\draw[-{latex}] (7-7) -- (6-6);
\draw[-{latex}] (5-7) [bend right=25]to (7-7);
\draw[-{latex}] (6-6) -- (6-7);

\end{tikzpicture}
\caption{The reduced Banff algorithm applied to the quiver $Q(CG,3)$.}
\label{fig:Banff}
\end{figure}

\section{Relationship with reddening}
\label{s:reddening}

A maximal green sequence or reddening sequence is a special sequence of mutations whose existence gives rise to additional properties of the underlying cluster algebra. They are mutation sequences which take the framed quiver of the cluster algebra to the coframed quiver of it (in the case of maximal green sequences, the mutations must be green mutations). For a more detailed introduction to these sequences see the work of Brustle, Dupont, and Perotin \cite{BDP}. 

It has been observed in the literature that the existence of a maximal green sequence or reddening sequence seems to coincide with equality of the cluster algebra and upper cluster algebra~\cite{CLS}. The existence of such a sequence depends only on the mutable part of a quiver. 

The quiver $Q(CG,3)$ provides an interesting case regarding the connection between these mutation sequences and the upper cluster algebra.
The mutable part of the quiver $Q(CG,3)$ with vertices labeled is shown in Figure~\ref{fig:MGS}.
With this labeling of vertices it can be checked that $(2,3,4,1,5,1,2,6,3)$ is a maximal green sequence.
Hence, we see the relationship of reddening sequences and equality of the cluster algebra and upper cluster algebra is again sensitive to the choice of ground ring since $\A_{\bbZ\bbP_+}(CG,3) \neq \U_{\bbZ\bbP_+}(CG,3)$ but $\A_{\bbZ\bbP}(CG,3) = \U_{\bbZ\bbP}(CG,3)$.
The maximal green sequence exhibited for $Q(CG,3)$ can be found using a new technique called \emph{component preserving mutations}.
An explanation of the idea of component preserving mutations along with further examples is given in~\cite{CP}.

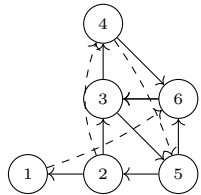
\begin{figure}
\begin{tikzpicture}
\node at (2,3) [circle, draw] (23) {\tiny{$4$}};
\node at (2,2) [circle, draw] (22) {\tiny{$3$}};
\node at (3,2) [circle, draw] (32) {\tiny{$6$}};
\node at (1,1) [circle, draw] (11) {\tiny{$1$}};
\node at (2,1) [circle, draw] (21) {\tiny{$2$}};
\node at (3,1) [circle, draw] (31) {\tiny{$5$}};

\draw[->] (21) -- (11);
\draw[->] (21) -- (22);
\draw[->] (32) -- (22);
\draw[->] (22) -- (31);
\draw[->] (31) -- (32);
\draw[->] (31) -- (21);
\draw[->] (22) -- (23);
\draw[->] (23) -- (32);
\draw[->] (32) -- (22);

\draw[->, dashed] (11) [bend right=8]to (32);
\draw[->, dashed] (21) [bend left=20]to (23);
\draw[->, dashed] (23) [bend left=8]to (31);

\end{tikzpicture}
\caption{The mutable part of the quiver $Q(CG,3)$.}
\label{fig:MGS}
\end{figure}

\section{Conclusion}
\label{sec:conclusion}

In this note we initiate a study of the upper cluster algebras and the dependence on ground rings.
Our main example showing the sensitivity of ground ring choice was the Cremmer-Gervais cluster structure.
In showing that $\A_{\bbZ\bbP_+}(CG,3) \neq \U_{\bbZ\bbP_+}(CG,3)$ Gekhtman, Shapiro, and Vainshtein prove that $x_{12} \in \U_{\bbZ\bbP_+}(CG,3)$ but $x_{12} \not\in \A_{\bbZ\bbP_+}(CG,3)$ where $x_{ij}$ denotes the matrix entry coordinate functions in $\Mat_3$.
However, we know that $x_{12} \in \A_{\bbZ\bbP}(CG,3)$ which leads to the following problem.
\begin{problem}
Find an expression for $x_{12}$ as an element of $\A_{\bbZ\bbP}$.
\end{problem}

Another interesting question regarding the example above is which additional elements of must be added to the ground ring in order to force $\U=\A$. 

\begin{problem}
Find ground rings $\bbA$, with $\bbZ\bbP_+ \subsetneq \bbA \subsetneq \bbZ\bbP$ where $\A_{\bbA}=\U_{\bbA}$.
\label{problem:ring}
\end{problem}
This problem can be attacked by running the Banff algorithm and tracking which coefficients are required to be invertible.
Though this may lead to a ground ring larger than necessary.
Similarly, a version of Problem~\ref{problem:ring} can be asked whenever one has a cluster algebra with $\A = \U$ over $\bbZ\bbP$.

As mentioned, a useful tool for showing $\A = \U$ over $\bbZ\bbP$ is Muller's theory of locally acyclic cluster algebras and Banff algorithm.
In Theorem~\ref{thm:A=U} we adapt ideas for the theory of locally acyclic cluster algebras to other ground rings.
We focus on the particular ground $\bbZ\bbP_+$ in Corollary~\ref{cor:A=Ugeo}.
Though currently the methods for concluding $\A = \U$ over other ground rings are not as robust as for $\bbZ\bbP$.
\begin{problem}
Find sufficient conditions for $\A = \U$ over a given ground ring $\bbA$.
\end{problem}

\begin{problem}
Find a version of the Banff algorithm over a given ground ring $\bbA$.
\label{problem:Banff}
\end{problem}

Of course we have a version of the Banff algorithm by simply defining a covering pair to be a pair such that the corresponding localizations cover the cluster algebra.
However, a solution to Problem~\ref{problem:Banff} should allow for a covering pair to be detected efficiently in terms of the quiver or exchange matrix.
Also a solution should allow one to show $\A_{\bbA} = \U_{\bbA}$ for some nice class of examples. 
\bibliographystyle{alpha}
\bibliography{upper}

\newcommand{\etalchar}[1]{$^{#1}$}
\begin{thebibliography}{{Sta}18}

\bibitem[BD82]{BD}
A.~A. Belavin and V.~G. Drinfeld.
\newblock Solutions of the classical {Y}ang-{B}axter equation for simple {L}ie
  algebras.
\newblock {\em Funktsional. Anal. i Prilozhen.}, 16(3):1--29, 96, 1982.

\bibitem[BDP14]{BDP}
Thomas Br\"ustle, Gr\'egoire Dupont, and Matthieu P\'erotin.
\newblock On maximal green sequences.
\newblock {\em Int. Math. Res. Not. IMRN}, (16):4547--4586, 2014.

\bibitem[BFZ05]{BFZ}
Arkady Berenstein, Sergey Fomin, and Andrei Zelevinsky.
\newblock Cluster algebras. {III}. {U}pper bounds and double {B}ruhat cells.
\newblock {\em Duke Math. J.}, 126(1):1--52, 2005.

\bibitem[BMR{\etalchar{+}}]{CP}
Eric Bucher, John Machacek, Evan Runberg, Abe Yeck, and Ethan Zewde.
\newblock Building maximal green sequences via component preserving mutations.
\newblock arXiv:1902.02262.

\bibitem[CLS15]{CLS}
Ilke Canakci, Kyungyong Lee, and Ralf Schiffler.
\newblock On cluster algebras from unpunctured surfaces with one marked point.
\newblock {\em Proc. Amer. Math. Soc. Ser. B}, 2:35--49, 2015.

\bibitem[Eis14]{E1}
Idan Eisner.
\newblock Exotic cluster structures on {$SL_5$}.
\newblock {\em J. Phys. A}, 47(47):474002, 23, 2014.

\bibitem[Eis17a]{E2}
Idan Eisner.
\newblock Exotic cluster structures on {$SL_n$} with {B}elavin-{D}rinfeld data
  of minimal size, {I}. {T}he structure.
\newblock {\em Israel J. Math.}, 218(1):391--443, 2017.

\bibitem[Eis17b]{E3}
Idan Eisner.
\newblock Exotic cluster structures on {$SL_n$} with {B}elavin-{D}rinfeld data
  of minimal size, {II}. {C}orrespondence between cluster structures and
  {B}elavin-{D}rinfeld triples.
\newblock {\em Israel J. Math.}, 218(1):445--487, 2017.

\bibitem[FWZ]{FWZ}
Sergey Fomin, Lauren Williams, and Andrei Zelevinsky.
\newblock Introduction to cluster algebras. chapters 1-3.
\newblock arXiv:1608.05735 [math.CO].

\bibitem[FZ02]{FZI}
Sergey Fomin and Andrei Zelevinsky.
\newblock Cluster algebras. {I}. {F}oundations.
\newblock {\em J. Amer. Math. Soc.}, 15(2):497--529, 2002.

\bibitem[GSV10]{GSV-book}
Michael Gekhtman, Michael Shapiro, and Alek Vainshtein.
\newblock {\em Cluster algebras and {P}oisson geometry}, volume 167 of {\em
  Mathematical Surveys and Monographs}.
\newblock American Mathematical Society, Providence, RI, 2010.

\bibitem[GSV12]{GSV-BD}
M.~Gekhtman, M.~Shapiro, and A.~Vainshtein.
\newblock Cluster structures on simple complex {L}ie groups and
  {B}elavin-{D}rinfeld classification.
\newblock {\em Mosc. Math. J.}, 12(2):293--312, 460, 2012.

\bibitem[GSV14]{GSV-CGshort}
Michael Gekhtman, Michael Shapiro, and Alek Vainshtein.
\newblock Cremmer-{G}ervais cluster structure on {$SL_n$}.
\newblock {\em Proc. Natl. Acad. Sci. USA}, 111(27):9688--9695, 2014.

\bibitem[GSV17]{GSV-CGlong}
M.~Gekhtman, M.~Shapiro, and A.~Vainshtein.
\newblock Exotic cluster structures on {$SL_n$}: the {C}remmer-{G}ervais case.
\newblock {\em Mem. Amer. Math. Soc.}, 246(1165):v+94, 2017.

\bibitem[MO17]{log-can}
John Machacek and Nicholas Ovenhouse.
\newblock Log-canonical coordinates for {P}oisson brackets and rational changes
  of coordinates.
\newblock {\em J. Geom. Phys.}, 121:288--296, 2017.

\bibitem[Mul13]{MullerAdv}
Greg Muller.
\newblock Locally acyclic cluster algebras.
\newblock {\em Adv. Math.}, 233:207--247, 2013.

\bibitem[Mul14]{MullerSIGMA}
Greg Muller.
\newblock {$\mathcal A=\mathcal U$} for locally acyclic cluster algebras.
\newblock {\em SIGMA Symmetry Integrability Geom. Methods Appl.}, 10:Paper 094,
  8, 2014.

\bibitem[Sco06]{Scott}
Joshua~S. Scott.
\newblock Grassmannians and cluster algebras.
\newblock {\em Proc. London Math. Soc. (3)}, 92(2):345--380, 2006.

\bibitem[{Sta}18]{Stacks}
The {Stacks Project Authors}.
\newblock \textit{Stacks Project}.
\newblock \url{http://stacks.math.columbia.edu}, 2018.

\end{thebibliography}

\end{document}